\documentclass[12pt]{amsart}
\usepackage{amssymb,latexsym}
\usepackage{enumerate}
\usepackage{float}

\usepackage{mathtools}
\usepackage{amsfonts}
\usepackage{amssymb}
\usepackage{amsmath}
\usepackage{graphicx}
\usepackage{graphicx,color}
\usepackage[normalem]{ulem}
\usepackage{todonotes}

\numberwithin{equation}{section}
\makeatletter
\@namedef{subjclassname@2010}{%
  \textup{2010} Mathematics Subject Classification}
\makeatother

\usepackage{amsmath,amstext,amsthm,amsxtra,mathtools,amssymb}

\usepackage[bookmarksopen,bookmarksdepth=3,colorlinks,citecolor=red,pagebackref,hypertexnames=true]{hyperref}
\usepackage[backrefs,msc-links,nobysame,initials,non-sorted-cites]{amsrefs}

\usepackage[backrefs,msc-links,nobysame,initials]{amsrefs}

\allowdisplaybreaks

\newtheorem{thm}{Theorem}

\newtheorem{lem}{Lemma}

\theoremstyle{definition}

\theoremstyle{remark}

\newtheorem{con}{Conjecture}

\begin{document}

\title[A generalization of the Birkhoff Ergodic Theorem]{A theorem of Besicovitch and a generalization of the Birkhoff Ergodic Theorem}
\author[P. Hagelstein]{Paul Hagelstein}
\address{P. H.: Department of Mathematics, Baylor University, One Bear Place \#97328, Waco, Texas 76798-7328}
\email{\href{mailto:paul_hagelstein@baylor.edu}{paul\!\hspace{.018in}\_\,hagelstein@baylor.edu}}
\thanks{P. H. is partially supported by a grant from the Simons Foundation (\#521719 to Paul Hagelstein).}
\author[D. Herden]{Daniel Herden}
\address{D. H.: Department of Mathematics, Baylor University, One Bear Place \#97328, Waco, Texas 76798-7328}
\email{\href{mailto:daniel_herden@baylor.edu}{daniel\!\hspace{.018in}\_\,herden@baylor.edu}}
\author[A. Stokolos]{Alexander Stokolos}
\address{A. S.: Georgia Southern University, Department of Mathematical Sciences,  203 Georgia Avenue,
 Statesboro, Georgia 30460-8093}
\email{\href{mailto:astokolos@georgiasouthern.edu}{astokolos@georgiasouthern.edu}}


\begin{abstract}
A remarkable theorem of Besicovitch is that an integrable function $f$ on $\mathbb{R}^2$ is strongly differentiable if and only if its associated strong maximal function $M_S f$ is finite a.e.  We provide an analogue of Besicovitch's result in the context of ergodic theory that provides a generalization of Birkhoff's Ergodic Theorem.  In particular, we show that if   $f$ is a measurable function on a standard probability space and $T$ is an invertible measure-preserving transformation on that space, then the ergodic averages of $f$ with respect to $T$ converge a.e. if and only if the associated ergodic maximal function $T^*f$ is finite a.e.
\end{abstract}

\maketitle
\
\section{Introduction}

Let $f$ be an integrable function on $\mathbb{R}^2$.  The \emph{Lebesgue Differentiation Theorem} tells us that for a.e. $x \in \mathbb{R}^2$ the averages of $f$ over disks shrinking to $x$ tend to $f(x)$ itself.  More precisely, we have that
$$\lim_{r \rightarrow 0} \frac{1}{|B(x,r)|} \int_{B(x,r)}f = f(x)\;\;\;\textup{a.e.},$$
where $B(x,r)$ denotes the open disk centered at $x$ of radius $r$ and $|B(x,r)|$ denotes the area of that disk.  For a proof of this result, the reader is encouraged to consult \cite{Stein}.

The issue of the averages of $f$ over \emph{rectangles} shrinking to $x$ is more subtle.
 There do exist integrable functions $f$ on $\mathbb{R}^2$ such that, for a.e. $x \in \mathbb{R}^2$, there exists a sequence of rectangles $\{R_{x,j}\}$ shrinking toward $x$ for which
$$\lim_{j \rightarrow \infty} \frac{1}{|R_{x,j}|} \int_{R_{x,j}}f$$
fails to converge. Such functions $f$ can even be characteristic functions of sets!   (See \cite{Gu} for a nice exposition of this result.) This result is closely related to the well-known \emph{Kakeya Needle Problem}, and the interested reader is  encouraged to consult \cite{falconer} for more information on this topic.

If we restrict the class of rectangles that we allow ourselves to average over, we obtain better results.  In \cite{JMZ}, Jessen, Marcinkiewicz, and Zygmund proved that if $\mathcal{B}_2$ consists of all the open rectangles in $\mathbb{R}^2$ \emph{whose sides are parallel to the coordinate axes}, then for any function $f \in L^p(\mathbb{R}^2)$ with $1 < p \leq \infty$ one has
$$\lim_{j \rightarrow \infty}\frac{1}{|R_{j}|}\int_{R_{j}} f \;=\;f(x)\;\;$$
for a.e. $x \in \mathbb{R}^2$, where here $\{R_j\}$ is any sequence of rectangles in $\mathcal{B}_{2}$ shrinking toward $x$.  Such a function $f$ is said to be \emph{strongly differentiable}. Jessen, Marcinkiewicz, and Zygmund proved this by showing that the \emph{strong maximal operator} $M_S$, defined by
$$M_S f(x) = \sup_{x \in R \in \mathcal{B}_2}\frac{1}{|R|} \int_R |f|\;,$$
satisfies for every $1 < p < \infty$ the \emph{weak type} $(p,p)$ estimate
$$\left|\{x \in \mathbb{R}^2 : M_Sf(x) > \alpha\}\right| \leq C_p \left(\frac{\|f\|_{L^p}}{\alpha}\right)^p\;.$$

Most mathematicians interested in multiparameter harmonic analysis are well aware of the above result.   Less well-known is a remarkable theorem that happens to be the paper  in \emph{Fundamenta Mathematicae} immediately preceding the famous paper of Jessen, Marcinkiewicz, and Zygmund.  In this paper \cite{besicovitch}, \emph{On differentiation of Lebesgue double integrals}, Besicovitch proved the following.

\begin{thm}[Besicovitch]\label{thm1}

Let $f$ be an integrable function on $\mathbb{R}^2$.
If $M_S f$ is finite a.e., then for a.e. $x$ we have $$\lim_{j \rightarrow \infty}\frac{1}{|R_j|}\int_{R_j}f = f(x)$$ whenever $\{R_j\}$ is a sequence of sets in $\mathcal{B}_2$ shrinking to $x$.
\end{thm}

Of course, if $f \in L^p(\mathbb{R}^2)$ for $1 < p < \infty$, the quantitative weak type $(p,p)$ bound satisfied by $M_S$ implies that $M_S f $ will be finite a.e.  It is for this reason that this paper of Besicovitch has received comparatively little attention.  However, it is worth observing that the above theorem of Besicovitch provides a means for obtaining a.e. differentiability results that bypasses the need for finding quantitative weak type bounds on the associated maximal operator.

 Many results in the study of differentiation of integrals have an analogous result in ergodic theory; for instance the Lebesgue Differentiation Theorem is structurally very similar to that of the \emph{Birkhoff Ergodic Theorem} on integrable functions.  This observation may be found at least as far back as the work of Wiener \cite{wiener1939}.   In that regard, we consider what the companion result of Besicovitch's Theorem might be when replacing the strong maximal operator $M_S$ by an ergodic maximal operator.  We are led immediately to the following theorem.

 \begin{thm}\label{thm2}
 Let $T$ be an invertible measure-preserving transformation on the standard probability space $(X, \Sigma, \mu)$ and let $f$ be a $\mu$-measurable function on that space.  If $T^\ast f(x)$ is finite $\mu$-a.e., where $T^\ast f$ is the ergodic maximal function defined by
 $$T^\ast f(x) = \sup_{n \ge 1} \frac{1}{n} \left|\sum_{k = 0}^{n-1} f(T^k x)\right|\;,$$
 then the limit
 $$\lim_{n \rightarrow \infty} \frac{1}{n}\sum_{k = 0}^{n-1} f(T^k x)$$
 exists $\mu$-a.e.
 \end{thm}

We remark that if $f$ is integrable, then by the Birkhoff Ergodic Theorem the above limit automatically holds a.e.
\\

The purpose of this paper is to provide a proof of the above theorem.  Section 2 provides a proof of this theorem in the special case that $T$ is an ergodic transformation.  Section 3 provides a proof of the general case by means of the ergodic decomposition theorem.   In the last section we indicate further directions of research, both in ergodic theory as well as the theory of differentiation of integrals.
\\

We remark that our techniques in Section 2 are strongly influenced by the work of Aaronson.  In fact, the key lemma of this section is stated without proof in \cite{aaronson} and
as Exercise 2.3.1 in
\cite{aaronsonbook}. For completeness, we  provide a proof, especially as it may be beneficial for harmonic analysts reading the paper without an extensive background in ergodic theory.
\\

It is our pleasure to thank Jon Aaronson for his helpful comments and advice regarding this paper.

\section{Maximal Functions Associated to Ergodic Transformations}

The purpose of this section is to state and prove the following lemma.

\begin{lem}\label{l1}

Let $T$ be an ergodic measure-preserving transformation on the probability space $(X, \Sigma, \mu)$ and let $f$ be a $\mu$-measurable function on that space.  If $T^\ast f(x)$ is finite $\mu$-a.e., where $T^\ast f$ is the ergodic maximal function defined by
 $$T^\ast f(x) = \sup_{n \ge 1} \frac{1}{n} \left|\sum_{k = 0}^{n-1} f(T^k x)\right|\;,$$
 then the limit
 $$\lim_{n \rightarrow \infty} \frac{1}{n}\sum_{k = 0}^{n-1} f(T^k x)$$
 exists $\mu$-a.e.
\end{lem}

\begin{proof}




Since $T^\ast f(x) < \infty$ $\mu$-a.e., we can choose some $M\ge 0$ and  a set $B \in \Sigma$ such that $\mu(B)>0$ 
with $T^\ast f(x) \le M$ for all $x \in B$. Letting $$f_n(x) = \sum_{k=0}^{n-1}f(T^kx)\;,$$
we have
\begin{equation}\label{e0}
\frac{|f_n(x)|}{n}\le M\;\; \mbox{for all}\;\; x \in B\;\; \mbox{and}\;\; n\ge 1.
\end{equation}

The goal now is to show that $\frac{f_n(x)}{n}$ converges to a finite constant $\mu$-a.e. on $B$,
in which case the $T$-invariant $\mu$-measurable set
$$\left\{x\in X : \frac{f_n(x)}{n}\;\; \mbox{converges for}\;\; n \rightarrow \infty\right\}$$
contains a subset of positive measure. As $T$ is ergodic, this would complete the proof.\\

The integer $$\phi(x) = \inf\{n \geq 1 : T^nx \in B\}$$ is defined for $\mu$-a.e. $x \in B$.
Let $T_B$ be the induced transformation of $T$ on $B$ given by
$$T_B(x) = T^{\phi(x)}x\;.$$  Note $T_B$ is an ergodic measure-preserving transformation on \mbox{$(B, \Sigma \cap B, \mu)$} and $\int_B \phi\; d\mu = 1$. (See \cite{kac1947, kakutani1943, petersen} in this regard.)

Define now the function $g$ on $B$ by
$$g(x) = f_{\phi(x)}(x)\;.$$
Moreover define the functions $\phi_n, g_n$ on $B$ by
$$\phi_n(x) = \sum_{k=0}^{n-1}\phi(T_B^k x)\;\; \mbox{and}\;\; g_n(x) = \sum_{k=0}^{n-1}g(T_B^k x)\;.$$
Introducing the measure $\mu_B= \mu/ \mu(B)$, we have the probability space $(B, \Sigma \cap B, \mu_B)$,
and applying the Birkhoff Ergodic Theorem to the integrable function $\phi$, we have
\begin{equation}\label{e1}
\lim_{n \rightarrow \infty}\frac{\phi_n(x)}{n} = \int_B \phi \;d\mu_B= \frac 1{\mu(B)} \int_B \phi \;d\mu=\frac 1{\mu(B)}\;\;\mu\textup{-a.e. on}\; B\;.
\end{equation}
In particular, $\mu$-a.e. on $B$ holds
\begin{equation}\label{e2}
\frac{\phi(T_B^n x)}{n} = \frac{\phi_{n+1}(x)-\phi_n(x)}{n} = \frac{n+1}n \cdot \frac{\phi_{n+1}(x)}{n+1} -  \frac{\phi_{n}(x)}{n}\rightarrow 0\;.
\end{equation}
Note that $g$ is a measurable function with $|g(x)| \leq M \phi(x)$ on $B$. Thus $\int_B |g| \;d\mu\le M< \infty$, and $g$ is integrable.
Hence by the Birkhoff Ergodic Theorem there exists a finite constant $c$ such that
\begin{equation}\label{e3}
\lim_{n \rightarrow \infty}\frac{g_n(x)}{n} = c \;\;\mu\textup{-a.e. on}\;\; B\;.
\end{equation}

For the remainder of this proof, fix any $x \in B$ such that (\ref{e2}) and (\ref{e3}) hold. For $n\ge 0$, we define the integers $k_n=k_n(x)$ and $l_n=l_n(x)$ such that
$$n = \phi_{k_n}(x) + l_n\;\;\textup{with}\;\;0 \leq l_n < \phi(T_B^{k_n}x)\;.$$
Note that $k_n$ and $l_n$ are uniquely determined with $k_n \rightarrow \infty$ for $n \rightarrow \infty$.
With \eqref{e2}, we have the estimate
$$0 \le \frac{l_n}{k_n} \le \frac{\phi(T_B^{k_n}x)}{k_n} \rightarrow 0\;\; \mbox{for}\;\; n\rightarrow \infty\;.$$
Hence
\begin{equation}\label{e4}
\lim_{n \rightarrow \infty}\frac{l_n}{k_n} = 0\;.
\end{equation}
With \eqref{e1}, we have
\begin{equation}\label{e5}
\lim_{n \rightarrow \infty}\frac{n}{k_n} = \lim_{n \rightarrow \infty}\frac{\phi_{k_n}(x)}{k_n}+\lim_{n \rightarrow \infty}\frac{l_n}{k_n}= \lim_{n \rightarrow \infty}\frac{\phi_{n}(x)}{n}=\frac 1{\mu(B)}\;.
\end{equation}
Observe that
$$f_n(x) = f_{\phi_{k_n}(x)}(x) + f_{l_n}(T_B^{k_n}x) = g_{k_n}(x) + f_{l_n}(T_B^{k_n}x)\;.$$
With \eqref{e0}, \eqref{e4}, and \eqref{e5}, we have the estimate
$$0 \le \frac{|f_{l_n}(T_B^{k_n}x)|}{n} \le \frac{M l_n}{n} = M\cdot \frac{k_n}n \cdot \frac{l_n}{k_n} \rightarrow 0\;\; \mbox{for}\;\; n\rightarrow \infty\;.$$
Together with \eqref{e3} and \eqref{e5}, we conclude
\begin{align}
\lim_{n \rightarrow \infty}\frac{f_n(x)}{n} &= \lim_{n \rightarrow \infty}\frac{g_{k_n}(x)}{n} + \lim_{n \rightarrow \infty}\frac{f_{l_n}(T_B^{k_n}x)}{n} \notag
\\ &= \lim_{n \rightarrow \infty}\frac{g_{k_n}(x)}{n}= \lim_{n \rightarrow \infty}\frac{g_{k_n}(x)}{k_n} \frac{k_n}n= c\mu(B)\;.\qedhere \notag
\end{align}
As this holds for all $x$ satisfying (\ref{e2}) and (\ref{e3}), and both of these hold for $\mu$-a.e. $x$ on $B$, the theorem follows.

\end{proof}

\section{Maximal Functions Associated to Measure-Preserving Transformations}
The purpose of this section is to provide a proof of  Theorem \ref{thm2} by using Lemma \ref{l1} combined with the Ergodic Decomposition Theorem.

The version of the Ergodic Decomposition Theorem we use follows from Theorem 2.2.9 in \cite{aaronsonbook} and is stated as follows:

\begin{thm}[Ergodic Decomposition Theorem]
Let $T$ be an invertible measure-preserving transformation on a standard probability space $(X, \Sigma, \mu)$.   Then there is a probability space $(Y, \Lambda, \eta)$ and a collection of probabilities
$$\{\mu_y : y \in Y\}$$
on $(X, \Sigma)$ such that
\begin{itemize}
\item[(i)] for $y \in Y$, $T$ is an invertible measure-preserving ergodic transformation of $(X, \Sigma, \mu_y)$, and
\item[(ii)] for $A \in \Sigma$, the map $y \mapsto \mu_y(A)$ is measurable, with
$$\mu(A) = \int_Y \mu_y(A)d\eta(y)\;.$$
\end{itemize}
\end{thm}
\begin{proof}[Proof of Theorem~\ref{thm2}]
Let $f$ be a $\mu$-measurable function on $X$ and suppose that $T^\ast f (x)$ is finite for $\mu$-a.e. $x$.   Let $A$ be the set of points in $X$ such that
$$\lim_{n \rightarrow \infty} \frac{1}{n}\sum_{k = 0}^{n-1} f(T^k x)$$
exists.   Note that $A$ is indeed $\mu$-measurable, being the complement of the set
$$\bigcup_{\alpha, \beta \in \mathbb{Q}} \bigg\{x \in X : \liminf_{n \rightarrow \infty}  \frac{1}{n}\sum_{k = 0}^{n-1} f(T^k x) < \alpha < \beta < \limsup_{n \rightarrow \infty}  \frac{1}{n}\sum_{k = 0}^{n-1} f(T^k x)\bigg\}\,.$$

It suffices to show that $\mu(A) = 1$.  If $\mu(A) < 1$, by the ergodic decomposition above there would exist $y \in Y$ such that $\mu_y(A) < 1$.  However, as $T$ is an ergodic transformation on the space $(X, \Sigma, \mu_y)$, by Lemma \ref{l1} we would have $\mu_y(A) = 1$, a contradiction.
\end{proof}

{\it{Remark:}}  The condition of invertibility of $T$ in Theorem \ref{thm2} enables the use of the version of the Ergodic Decomposition Theorem we provide here.  The conclusion of Theorem \ref{thm2} holds whenever $T$ and $(X, \Sigma, \mu)$ permit a decomposition as in the conclusion of the Ergodic Decomposition Theorem.

\section{Future Directions}
The theorem of Besicovitch and its analogue in the context of ergodic theory do suggest the following future directions of research, some related to very recent work of Hagelstein and Parissis \cite{hpfund2018}.
\\

{\bf{Problem:}}   \emph{Differentiation} of a function $f$ relates to averages of $f$ over sets of arbitrarily small diameter, whereas the strong maximal operator $M_S$ involves rectangles of any size.  This suggests that we might be able to strengthen Theorem \ref{thm1} by the following:
\\

Given a collection $\mathcal{B} = \{R_j\}$ of open sets in $\mathbb{R}^n$  we define the maximal operator $M_\mathcal{B}f$ by
$$M_{\mathcal{B}}f(x) = \sup_{x \in R \in \mathcal{B}}\frac{1}{|R|}\int_R |f|\;.$$
For $r > 0$ we set
$$\mathcal{B}_r = \{R \in \mathcal{B} : \textup{diam}\, R < r\}\;.$$
Define the maximal operator $\tilde{M}_\mathcal{B}$ by
$$\tilde{M}_\mathcal{B} f(x) = \lim_{r \rightarrow 0}M_{\mathcal{B}_r}f(x)\;.$$
\begin{con}
If $\mathcal{B}_2$ is the collection of rectangles in $\mathbb{R}^2$ whose sides are parallel  to the coordinate axes and $\tilde{M}_\mathcal{B}f(x)$ is finite a.e., then $f$ is strongly differentiable.
\end{con}

Of course, this conjecture may be generalized in many ways, e.g., if $\mathcal{B}$ is a translation invariant density basis of open sets in $\mathbb{R}^n$ and $\tilde{M}_\mathcal{B}f(x)$ is finite a.e., then $\mathcal{B}$ differentiates $f$.
\\

{\bf{Problem:}}    Theorem \ref{thm1} was generalized by de Guzm\'an and Men\'arguez to encompass homothecy invariant Busemann-Feller bases associated to convex sets in $\mathbb{R}^n$ with a center of symmetry.   (For a proof one may consult Chapter IV of \cite{Gu}.)    It is natural to consider to what extent Theorem \ref{thm1} may be further generalized.   One possible generalization is provided by the following:

\begin{con}
Let $\mathcal{B}$ be a translation invariant density basis consisting of open sets in $\mathbb{R}^n$.  If $M_\mathcal{B}f(x) < \infty$ a.e., then
$$\lim_{j \rightarrow \infty}\frac{1}{|R_{j}|}\int_{R_{j}} f \;=\;f(x)\;\;a.e.$$ where the limit is over an arbitrary sequence of sets $\{R_j\}$ in $\mathcal{B}$ shrinking to $x$.
\end{con}

{\bf{Problem:}}
It would be natural to desire to obtain a \emph{multiparameter} analogue of Theorem \ref{thm2} in the spirit of previous work of Hagelstein and Stokolos in \cite{hs2012}  and Hagelstein and Parissis in  \cite{hpnyj2017}.  In particular we make the following conjecture.

\begin{con}
Let $U$ and $V$ be a nonperiodic collection of invertible measure-preserving transformations on a standard probability space $(X, \Sigma, \mu)$ and define the associated strong ergodic maximal operator by
$$M_{U,V}f(x) = \sup_{m,n \ge 1}\frac{1}{mn}\left|\sum_{j=0}^{m-1} \sum_{k=0}^{n-1} f(U^j V^k x)\right|\;.$$
If $M_{U,V}f(x) < \infty$ $\mu$-a.e., then
$$\lim_{m,n \rightarrow \infty} \frac{1}{mn}\sum_{j=0}^{m-1} \sum_{k=0}^{n-1} f(U^j V^k x) \;\;\textup{converges $\mu$-a.e.}$$
\end{con}
Of course analogues of this conjecture exist where $(m,n)$ are allowed to only lie in a specified set $\Gamma$ in $\mathbb{Z}_+^2\;.$
\\

All of these topics are subjects of ongoing research.

\end{document}